\documentclass[11pt,reqno]{amsart}
\usepackage{amsaddr}
\usepackage{graphicx}
\usepackage{mathtools}
\usepackage{pgfplots}
\usepackage{orcidlink}
\usepackage{geometry}
\usepackage{stix}
\usepackage{caption}
\usepackage{amsmath,amsthm}
\usepackage{cite}
\usepackage{hyperref}
\hypersetup{
    colorlinks=true, 
    linktoc=all,     
    allcolors=blue,  
}

\newtheorem{theorem}{Theorem}[section]
\newtheorem{lemma}{Lemma}[section]
\newtheorem{proposition}{Proposition}[section]
\newtheorem{corollary}{Corollary}[section]

\theoremstyle{definition}

\newtheorem{example}{Example}[section]
\newtheorem{remark}{Remark}[section]
\numberwithin{equation}{section}
\pgfplotsset{compat=1.18}
\definecolor{qqzzqq}{rgb}{0.,0.6,0.}
\definecolor{zzttqq}{rgb}{0.6,0.2,0.}
\definecolor{wewdxt}{rgb}{0.45,0.45,0.45}
\definecolor{ududff}{rgb}{0.3,0.3,1.}

\begin{document}
\pagenumbering{arabic}

\title[A Generalized Euler Identity]{A Vector Generalization of Euler's Quadrilateral Theorem}
\author{Mohammad Hassan Murad\orcidlink{0000-0002-8293-5242}}
\address{Department of Mathematics\\
The University of Texas at Arlington, Arlington, TX, USA}
\email{mohammad.murad2@uta.edu}

\begin{abstract}
In this paper, we develop a unified algebraic framework for Euler-type identities in real and complex inner product spaces. Starting from the parallelogram identity, we derive Apollonius' identity and recover Euler's classical theorem. We then establish a general Euler-type identity valid for every finite collection of $n\geq 4$ vectors. The proof is based on a combinatorial analysis of pairwise distances. The resulting identity recovers Euler's theorem when $n=4$. Several previously known identities thus arise naturally within a single algebraic framework.
\end{abstract}

\keywords{Parallelogram identity; Apollonius' identity; Euler's theorem; quadrilaterals; inner product spaces}
\subjclass[2020]{Primary: 51M04, 51N20; Secondary: 46C05}


\maketitle

\section{Introduction}\label{sec:intro}
\noindent
Metric identities in Euclidean geometry often admit natural formulations in terms of vectors and inner products. Among the most fundamental are the parallelogram theorem and Apollonius' \,theorem. These identities relate distances between finite point configurations and reflect \,characteristic properties of inner product spaces.

\begin{theorem}[Parallelogram Theorem]\label{thm:classicpar}
Let $ABCD$ be a parallelogram. Then
\begin{equation}\label{eq:classpar}
2\bigl(|AB|^2+|BC|^2\bigr)=|AC|^2+|BD|^2.
\end{equation}
\end{theorem}

\begin{theorem}[Apollonius]\label{thm:classicapollo}
If $AD$ is a median of a triangle $ABC$, then
\[
|AB|^2+|AC|^2
=
2|AD|^2+\frac12|BC|^2.
\]
\end{theorem}

In 1748, Euler established a remarkable identity for quadrilaterals that relates the sums of the squares of the side lengths, the diagonals, and the segment joining the midpoints of the diagonals.

\begin{theorem}[Euler, 1748]\label{thm:classiceuler}
Let $ABCD$ be a quadrilateral, and let $L$ and $M$ be the midpoints of the diagonals $AC$ and $BD$, respectively. Then
\[
|AB|^2+|BC|^2+|CD|^2+|DA|^2
=
|AC|^2+|BD|^2+4|LM|^2.
\]
\end{theorem}

Euler's theorem may be regarded as a natural extension of the parallelogram identity. Indeed, a quadrilateral is a parallelogram if and only if the midpoints of its diagonals coincide, in which case Euler's theorem reduces precisely to \eqref{eq:classpar}.

Motivated by an exposition of Dunham \cite{Dunham2000}, Dence and Dence \cite{Dence-Dence2002} revisited Euler's theorem for convex quadrilaterals, while Kandall \cite{Kandall2002} gave a concise vector proof and observed that the \,identity remains valid for arbitrary point configurations in $\mathbb{R}^m$. See Figure \ref{fig:eulerinrm}. The same identity was \,established much earlier by Small \cite{Small1929} as an algebraic relation in the complex plane.

\begin{figure}
    \centering
\begin{tikzpicture}[scale=1]
\clip(-4,-6) rectangle (9,6);
\fill[line width=1.pt,color=zzttqq,fill=zzttqq,fill opacity=0.15] (0.,1.) -- (2.,3.) -- (0.,5.) -- (-2.,3.) -- cycle;
\fill[line width=1.pt,color=zzttqq,fill=zzttqq,fill opacity=0.15] (3.,-2.) -- (0.,1.) -- (-3.,-2.) -- (0.,-5.) -- cycle;
\fill[line width=1.pt,color=zzttqq,fill=zzttqq,fill opacity=0.15] (4.,2.) -- (3.,-2.) -- (7.,-3.) -- (8.,1.) -- cycle;
\fill[line width=1.pt,color=zzttqq,fill=zzttqq,fill opacity=0.15] (2.,3.) -- (4.,2.) -- (5.,4.) -- (3.,5.) -- cycle;
\fill[line width=1.pt,color=blue,fill=blue,fill opacity=0.15] (2.,3.) -- (3.,-2.) -- (8.,-1.) -- (7.,4.) -- cycle;
\fill[line width=1.pt,color=blue,fill=blue,fill opacity=0.15] (4.,2.) -- (0.,1.) -- (1.,-3.) -- (5.,-2.) -- cycle;
\fill[line width=1.pt,color=qqzzqq,fill=qqzzqq,fill opacity=0.15] (2.5,0.5) -- (2.,1.5) -- (1.,1.) -- (1.5,0.) -- cycle;
\coordinate (A) at (2,3);
\coordinate (B) at (0,1);
\coordinate (C) at (3,-2);
\coordinate (D) at (4,2);
\draw [line width=1.pt] (A)-- (B);
\draw [line width=1.pt] (B)-- (C);
\draw [line width=1.pt] (C)-- (D);
\draw [line width=1.pt] (D)-- (A);
\draw [line width=1.pt] (B)-- (D);
\draw [line width=1.pt] (A)-- (C);
\draw [line width=1.pt,color=qqzzqq] (2.5,0.5)-- (2,1.5);
\draw [line width=1.pt,color=zzttqq] (2.,3.)-- (0.,5.);
\draw [line width=1.pt,color=zzttqq] (0.,5.)-- (-2.,3.);
\draw [line width=1.pt,color=zzttqq] (-2.,3.)-- (0.,1.);
\draw [line width=1.pt,color=zzttqq] (0.,1.)-- (-3.,-2.);
\draw [line width=1.pt,color=zzttqq] (-3.,-2.)-- (0.,-5.);
\draw [line width=1.pt,color=zzttqq] (0.,-5.)-- (3.,-2.);
\draw [line width=1.pt,color=zzttqq] (3.,-2.)-- (7.,-3.);
\draw [line width=1.pt,color=zzttqq] (7.,-3.)-- (8.,1.);
\draw [line width=1.pt,color=zzttqq] (8.,1.)-- (4.,2.);
\draw [line width=1.pt,color=zzttqq] (4.,2.)-- (5.,4.);
\draw [line width=1.pt,color=zzttqq] (5.,4.)-- (3.,5.);
\draw [line width=1.pt,color=zzttqq] (3.,5.)-- (2.,3.);
\draw [line width=1.pt,color=blue] (3.,-2.)-- (8.,-1.);
\draw [line width=1.pt,color=blue] (8.,-1.)-- (7.,4.);
\draw [line width=1.pt,color=blue] (7.,4.)-- (2.,3.);
\draw [line width=1.pt,color=blue] (0.,1.)-- (1.,-3.);
\draw [line width=1.pt,color=blue] (1.,-3.)-- (5.,-2.);
\draw [line width=1.pt,color=blue] (5.,-2.)-- (4.,2.);
\draw [line width=1.pt,color=qqzzqq] (2.,1.5)-- (1.,1.);
\draw [line width=1.pt,color=qqzzqq] (1.,1.)-- (1.5,0.);
\draw [line width=1.pt,color=qqzzqq] (1.5,0.)-- (2.5,0.5);
\begin{scriptsize}
\draw [fill=ududff] (2.,3.) circle (2.0pt);
\draw[color=black] (1.921644474714935,3.3826245015010383) node {$A$};
\draw [fill=ududff] (0.,1.) circle (2.0pt);
\draw[color=black] (-0.35,1.) node {$B$};
\draw [fill=ududff] (3.,-2.) circle (2.0pt);
\draw[color=black] (3.1,-2.3) node {$C$};
\draw [fill=ududff] (4.,2.) circle (2.0pt);
\draw[color=black] (4.3,2.1829561284596783) node {$D$};
\draw [fill=wewdxt] (2.5,0.5) circle (2.0pt);
\draw[color=black] (2.8,0.65) node {$L$};
\draw [fill=wewdxt] (2.,1.5) circle (2.0pt);
\draw[color=black] (2,1.8) node {$M$};
\draw [fill=wewdxt] (0.,5.) circle (2.0pt);
\draw [fill=wewdxt] (-2.,3.) circle (2.0pt);
\draw [fill=wewdxt] (-3.,-2.) circle (2.0pt);
\draw [fill=wewdxt] (0.,-5.) circle (2.0pt);
\draw [fill=wewdxt] (7.,-3.) circle (2.0pt);
\draw [fill=wewdxt] (8.,1.) circle (2.0pt);
\draw [fill=wewdxt] (5.,4.) circle (2.0pt);
\draw [fill=wewdxt] (3.,5.) circle (2.0pt);
\draw [fill=wewdxt] (8.,-1.) circle (2.0pt);
\draw [fill=wewdxt] (7.,4.) circle (2.0pt);
\draw [fill=wewdxt] (1.,-3.) circle (2.0pt);
\draw [fill=wewdxt] (5.,-2.) circle (2.0pt);
\draw [fill=wewdxt] (1.,1.) circle (2.0pt);
\draw [fill=wewdxt] (1.5,0.) circle (2.0pt);
\end{scriptsize}
\end{tikzpicture}
    \caption{Illustration of Euler's theorem in 
$\mathbb R^m$. The combined areas of the squares on the sides $\overline{AB}$, $\overline{BC}$, $\overline{CD}$, and $\overline{DA}$ equal the combined areas of the squares on the diagonals $\overline{AC}$, $\overline{BD}$ together with four times the area of the square on the \,segment $\overline{LM}$ joining the midpoints of the diagonals. The squares are drawn schematically; in $\mathbb R^m$, they need not lie in a common plane.}
    \label{fig:eulerinrm}
\end{figure}

These developments demonstrate that Euler's theorem is fundamentally an algebraic identity in an inner product space rather than merely a geometric statement about planar quadrilaterals.

An algebraic extension of the parallelogram identity to an even number of vectors was obtained by Douglas \cite{Douglas1981}. To the best of the author's knowledge, the first extensions of Euler's theorem beyond quadrilaterals were derived by Rassias \cite{Rassias2006}, who obtained identities for configurations of five and six vectors by repeated applications of Euler's theorem. More recently, Dlab and Williams \cite{Dlab-Williams2019} proposed further generalizations to larger configurations. However, these extensions were obtained separately for different values of $n$ and do not arise as special cases of a single algebraic identity valid for all $n\ge4$.

The purpose of the present paper is to establish such a unified identity. Our main result, Theorem \ref{thm:geneul}, provides an Euler-type identity for every finite collection of $n\ge4$ vectors in a real or complex inner product space. The proof is based on a combinatorial analysis of alternating vector sums and pairwise distances.

The paper is organized as follows. In Section \ref{sec:geneul} we prove the main result and provide its geometric interpretations. Section \ref{sec:cengenapoll} is devoted to generalized parallelogram and Apollonius-type identities together with several centroid formulas. The final section contains concluding remarks.

Throughout the paper, $\mathbb K\in\{\mathbb R,\mathbb C\}$ and $V$ denotes an inner product space over $\mathbb K$ with inner product $(\cdot,\cdot)$ and induced norm
\[
\|x\|=\sqrt{(x,x)}.
\]

\section{From the Parallelogram Identity to Euler's Theorem}
In this section we rewrite the parallelogram theorem, the Apollonius' theorem and the Euler's theorem in an inner product space.

\begin{theorem}[Parallelogram Identity]\label{thm:pariden}
    Let $x,y$ be vectors in an inner product space $V$. Then
    \begin{align}\label{eq:pariden}
        \Vert x+y \Vert^2+\Vert x-y \Vert^2&=2(\Vert x \Vert^2+\Vert y \Vert^2).
    \end{align}
\end{theorem}
The proof of Theorem \ref{thm:pariden} is an easy exercise in linear algebra (see, for example, \cite{Axler2024}). The parallelogram identity \eqref{eq:pariden} may be used to establish \emph{Apollonius' identity} \eqref{eq:apollonius}.
\begin{theorem}[Apollonius]\label{thm:apollonius}
    Let $x,y,z$ be vectors in an inner product space $V$. Then
    \begin{equation}\label{eq:apollonius}
        \Vert x-z \Vert^2+\Vert y-z \Vert^2=2\left\Vert z-\frac{x+y}{2} \right\Vert^2+\frac{1}{2}\Vert x-y \Vert^2.
    \end{equation}
\end{theorem}
\begin{proof} Applying the parallelogram identity \eqref{eq:pariden} to the vectors $x-z$ and $y-z$, we obtain
\begin{equation}\label{eq:1.3}
        \Vert x-z+y-z \Vert^2+\Vert x-y \Vert^2=2(\Vert x-z \Vert^2+\Vert y-z \Vert^2).
    \end{equation}
Dividing \eqref{eq:1.3} by 2 yields \eqref{eq:apollonius}.
\end{proof}

Repeated applications of Apollonius' identity yield Euler's theorem in an inner product space. 
\begin{theorem}[Euler]\label{thm:euler}
    Let $x_1,x_2,x_3,x_4$ be four vectors in an inner product space $V$. Then
    \begin{equation}\label{eq:euler}
    \Vert x_1-x_2 \Vert^2+\Vert x_2-x_3 \Vert^2+\Vert x_3-x_4 \Vert^2+\Vert x_4-x_1 \Vert^2=\Vert x_1-x_3 \Vert^2+\Vert x_2-x_4 \Vert^2+4\left\Vert \frac{x_1+x_3}{2}-\frac{x_2+x_4}{2} \right\Vert^2.
\end{equation}
\end{theorem}
\begin{proof} We denote the midpoint vector by
\begin{equation*}
    \bar{x}_{ij}:=\frac{1}{2}(x_i+x_j).    
\end{equation*}
Applying Apollonius' identity to the triples $x_1,x_2,x_4$ and $x_2,x_3,x_4$ yield
\begin{align}
    \Vert x_1-x_2\Vert^2+\Vert x_1-x_4\Vert^2&=2\Vert x_1-\bar{x}_{24}\Vert^2+\frac{1}{2}\Vert x_2-x_4\Vert^2 \label{eq:1.5}\\
    \Vert x_2-x_3\Vert^2+\Vert x_3-x_4\Vert^2&=2\Vert x_3-\bar{x}_{24}\Vert^2+\frac{1}{2}\Vert x_2-x_4\Vert^2.\label{eq:1.6}
\end{align}
Adding equations \eqref{eq:1.5} and \eqref{eq:1.6}, we obtain
\begin{equation}\label{eq:1.7}
    \Vert x_1-x_2 \Vert^2+\Vert x_2-x_3 \Vert^2+\Vert x_3-x_4 \Vert^2+\Vert x_1-x_4 \Vert^2=2\Vert x_1-\bar{x}_{24}\Vert^2+2\Vert x_3-\bar{x}_{24}\Vert^2+\Vert x_2-x_4 \Vert^2.
\end{equation}
Now applying Apollonius' identity to the vectors $x_1,x_3,\bar{x}_{24}$, we obtain
\begin{align}\label{eq:1.8}
    \Vert x_1-\bar{x}_{24}\Vert^2+\Vert x_3-\bar{x}_{24}\Vert^2&=2\Vert \bar{x}_{13}-\bar{x}_{24}\Vert^2+\frac{1}{2}\Vert x_1-x_3\Vert^2.
\end{align}
Using \eqref{eq:1.8} into \eqref{eq:1.7} one obtains the desired result.
\end{proof}

Conversely, Apollonius' identity is recovered from Euler's identity \eqref{eq:euler} by setting $x_2=x_4$.

\begin{remark}
Equation \eqref{eq:euler} may be rewritten in the equivalent form
\[
\Vert x_1-x_2 \Vert^2+\Vert x_2-x_3 \Vert^2+\Vert x_3-x_4 \Vert^2+\Vert x_4-x_1 \Vert^2
=
\Vert x_1-x_3 \Vert^2+\Vert x_2-x_4 \Vert^2
+\left\Vert x_1-x_2+x_3-x_4 \right\Vert^2.
\]
This formulation is particularly convenient for generalization. Indeed, it serves as the prototype for the Euler-type identities established in the next section for arbitrary collections of $n\ge 4$ vectors in an inner product space.
\end{remark}

\section{Generalized Euler Identity}\label{sec:geneul}

\subsection{Main Result}\label{sec:geneulmid}

We begin with a combinatorial lemma that will be used to establish the main theorem of the paper. 
\begin{lemma}\label{lemm:coefxp}
Let $n,k\in \mathbb{N}$ with $4\leq k\leq n$, and let $i_j \in \{1,2,...,n\}$ for $j\in \{1,2,...,k\}$. 
\begin{itemize}
    \item[(a)] The coefficient of $\Vert x_p \Vert^2$, $p\in \{i_1,...,i_k\}$, in the expansion of
\[
\sum_{1\leq i_1<...<i_k\leq n}\Vert x_{i_1}-x_{i_2}+\ldots (-1)^{k-1} x_{i_k} \Vert^2
\]
is
\[
\binom{n-1}{k-1}.
\]
    \item[(b)] The coefficient of $\mathrm{Re}(x_p,x_q)$ where $p,q\in \{i_1,...,i_k\}$ with $p<q$, in the expansion of
\[
\sum_{1\leq i_1<...<i_k\leq n}\Vert x_{i_1}-x_{i_2}+\ldots (-1)^{k-1} x_{i_k} \Vert^2
\]
is given by $2c_{pq}$ where
\begin{equation}\label{eq:cpq}
    c_{pq}=\sum_{r=0}^{k-2} (-1)^{r+1}\binom{n-q+p-1}{k-2-r}\binom{q-p-1}{r}.
\end{equation}
In particular, $c_{pq}=c_{1,q-p+1}$. And for $q=p+1$,
\begin{equation*}
    c_{pq}=-\binom{n-2}{k-2}.
\end{equation*} 
\end{itemize}
\end{lemma}
\begin{proof} 
To prove (a), suppose that $i_j=p$ for some $j\in\{1,2,\ldots,k\}$. Then there are
\[
\binom{n-p}{k-j}
\]
choices for $x_{i_{j+1}},\ldots,x_{i_k}$ for each of the
\[
\binom{p-1}{j-1}
\]
choices of $x_{i_1},\ldots,x_{i_{j-1}}$. Therefore, the coefficient of $\|x_p\|^2$ is
\[
\sum_{j=1}^{k}\binom{n-p}{k-j}\binom{p-1}{j-1}
=
\binom{n-1}{k-1},
\]
where the last equality follows from Vandermonde's identity
\[
\binom{m+n}{k}
=
\sum_{r=0}^{k}
\binom{m}{k-r}\binom{n}{r}.
\]

To prove (b), we proceed similarly. First consider the case $q-p=1$. There are $k-1$ subcases.

\begin{itemize}
\item
\[
\|x_p-x_q+x_{i_3}-x_{i_4}+\cdots+(-1)^{k-1}x_{i_k}\|^2,
\]
with
\[
1\le p<q<i_3<\cdots<i_k\le n.
\]
In this case there are $\binom{n-q}{k-2}$ choices for
$x_{i_3},\ldots,x_{i_k}$. Hence the contribution to the coefficient of
$\mathrm{Re}(x_p,x_q)$ is
\[
-2\binom{n-q}{k-2}.
\]

\item
\[
\|x_{i_1}-x_p+x_q-x_{i_4}+\cdots+(-1)^{k-1}x_{i_k}\|^2,
\]
with
\[
1\le i_1<p<q<i_4<\cdots<i_k\le n.
\]
In this case there are $\binom{n-q}{k-3}$ choices for
$x_{i_4},\ldots,x_{i_k}$ for each of the
$\binom{p-1}{1}$ choices of $x_{i_1}$. Thus the contribution to the coefficient of $\mathrm{Re}(x_p,x_q)$ is
\[
-2\binom{n-q}{k-3}\binom{p-1}{1}.
\]
\end{itemize}

Summing over all $k-1$ subcases gives
\[
-2\sum_{r=0}^{k-2}
\binom{n-q}{k-2-r}\binom{p-1}{r}
=
-2\binom{n-q+p-1}{k-2}.
\]

Next consider the case $q-p=2$. There are now $k-2$ subcases.

\begin{itemize}
\item
\[
\|x_p-x_{i_2}+x_q+x_{i_3}-\cdots+(-1)^{k-1}x_{i_k}\|^2,
\]
with
\[
1\le p<i_2<q<i_3<\cdots<i_k\le n.
\]
There are $\binom{n-q}{k-3}$ choices for
$x_{i_3},\ldots,x_{i_k}$ for each of the
$\binom{q-p-1}{1}$ choices of $x_{i_2}$. Hence the contribution to the coefficient of $\mathrm{Re}(x_p,x_q)$ is
\[
2\binom{n-q}{k-3}\binom{q-p-1}{1}.
\]

\item
\[
\|x_{i_1}-x_p+x_{i_3}-x_q+\cdots\|^2,
\]
with
\[
1\le i_1<p<i_3<q<i_4<\cdots<i_k\le n.
\]
There are $\binom{n-q}{k-4}$ choices for
$x_{i_4},\ldots,x_{i_k}$ for each of the
$\binom{q-p-1}{1}$ choices of $x_{i_3}$ and each of the
$\binom{p-1}{1}$ choices of $x_{i_1}$. Thus the contribution to the coefficient of $\mathrm{Re}(x_p,x_q)$ is
\[
2\binom{n-q}{k-4}
\binom{q-p-1}{1}
\binom{p-1}{1}.
\]
\end{itemize}

Summing over all $k-2$ subcases yields
\[
2\sum_{r=0}^{k-3}
\binom{n-q}{k-3-r}
\binom{q-p-1}{1}
\binom{p-1}{r}
=
2\binom{n-q+p-1}{k-3}
\binom{q-p-1}{1}.
\]

Proceeding in this manner, we obtain the coefficient of
$\mathrm{Re}(x_p,x_q)$ as
\[
2\sum_{r=0}^{k-2}
(-1)^{r+1}
\binom{n-q+p-1}{k-2-r}
\binom{q-p-1}{r}
=
2c_{pq}.
\]

This completes the proof.
\end{proof}

\begin{corollary}\label{cor:sumcij} Let $k=4$ and $c_{ij}$ be defined as in \eqref{eq:cpq}. Then
    \begin{equation}\label{eq:sumcij}
        \sum_{j=i+2}^{n+i-2}c_{ij}+\binom{n-1}{3}=2\binom{n-2}{2}.
    \end{equation}
\end{corollary}
\begin{proof} Using (b) of Lemma \ref{lemm:coefxp} with $k=4$ we find
    \begin{equation*}
            \sum_{j=i+2}^{n+i-2}c_{ij}=\sum_{j=3}^{n-1}c_{1j}
            =-2 j^2+4j + 2j n - \frac{1}{2} n^2- \frac{3}{2} n - 3.
    \end{equation*}
 Summing over $j$ gives \eqref{eq:sumcij}.   
\end{proof}

We are now ready to prove our main result.
\begin{theorem}[Generalized Euler Identity]\label{thm:geneul}
Let $n\geq 4$, and let $x_1,x_2,\ldots,x_n$ be vectors in an inner product space $V$. Let $c_{ij}$ be defined by \eqref{eq:cpq}. Then
\begin{equation}\label{eq:geneul}
\binom{n-2}{2}\sum_{i=1}^{n}\|x_i-x_{i+1}\|^2
=
\sum_{\substack{1\le i<j\le n\\1<j-i<n-1}}
c_{ij}\|x_i-x_j\|^2
+
\sum_{1\le i<j<k<l\le n}
\|x_i-x_j+x_k-x_l\|^2,
\end{equation}
where $x_{n+1}=x_1$.
\end{theorem}

\begin{proof}
We prove the result for a complex inner product space; the real case is identical.

Let $1\le i<j\le n$. By Lemma \ref{lemm:coefxp}, the coefficient of
$\mathrm{Re}(x_i,x_j)$ in
\[
\sum_{1\le i<j<k<l\le n}
\|x_i-x_j+x_k-x_l\|^2
\]
is $2c_{ij}$, whereas its coefficient in
\[
\sum_{\substack{1\le i<j\le n\\1<j-i<n-1}}
c_{ij}\|x_i-x_j\|^2
\]
is $-2c_{ij}$.

Hence these terms cancel whenever $1<j-i<n-1$. The only remaining mixed terms occur when $j=i+1$ (cyclically), in which case
\[
c_{i,i+1}
=
-\binom{n-2}{2},
\]
so the coefficient of $\mathrm{Re}(x_i,x_{i+1})$ is
\[
-2c_{i,i+1}
=
2\binom{n-2}{2}.
\]

Therefore,
\begin{align*}
&\sum_{1\le i<j<k<l\le n}
\|x_i-x_j+x_k-x_l\|^2
+
\sum_{\substack{1\le i<j\le n\\1<j-i<n-1}}
c_{ij}\|x_i-x_j\|^2 \\
&=
\sum_{i=1}^{n}
\left(
\sum_{j=i+2}^{n+i-2} c_{ij}
+
\binom{n-1}{3}
\right)\|x_i\|^2
-
2\sum_{i=1}^{n}
c_{i,i+1}\,\mathrm{Re}(x_i,x_{i+1}) \\
&=
2\binom{n-2}{2}\sum_{i=1}^{n}\|x_i\|^2
-
2\binom{n-2}{2}
\sum_{i=1}^{n}\mathrm{Re}(x_i,x_{i+1}) \\
&=
\binom{n-2}{2}
\sum_{i=1}^{n}\|x_i-x_{i+1}\|^2,
\end{align*}
where we used Corollary \ref{cor:sumcij} and the convention
$x_{n+1}=x_1$.
\end{proof}

\begin{corollary}\label{cor:geneulRm}
Let $A_1,A_2,\ldots,A_n$ be (not necessarily distinct) points in
$\mathbb{R}^m$, and let $M_{ij}$ denote the midpoint of the segment
$\overline{A_iA_j}$. If $c_{ij}$ is defined by \eqref{eq:cpq}, then
\begin{equation}\label{eq:geneulRm}
\binom{n-2}{2}\sum_{i=1}^{n}|A_iA_{i+1}|^2
=
\sum_{\substack{1\le i<j\le n\\1<j-i<n-1}}
c_{ij}|A_iA_j|^2
+
\sum_{1\le i<j<k<l\le n}
4|M_{ik}M_{jl}|^2,
\end{equation}
where $A_{n+1}=A_1$.

In particular, when $n=4$, \eqref{eq:geneulRm} reduces to Euler's quadrilateral theorem.
\end{corollary}

\begin{proof}
Let $x_i$ be the position vector of $A_i$. Since
\[
x_i-x_j+x_k-x_l
=
2(m_{ik}-m_{jl}),
\]
where $m_{ik}$ and $m_{jl}$ denote the position vectors of the midpoints
$M_{ik}$ and $M_{jl}$, respectively, it follows that
\[
\|x_i-x_j+x_k-x_l\|^2
=
4|M_{ik}M_{jl}|^2.
\]
Substituting this relation into \eqref{eq:geneul} yields
\eqref{eq:geneulRm}.

For $n=4$, the only admissible pairs of indices are $(i,k)=(1,3)$ and $(j,l)=(2,4)$, and a direct computation from \eqref{eq:cpq} gives
\[
c_{13}=c_{24}=1.
\]
Moreover, there is only one quadruple of indices $(i,j,k,l)=(1,2,3,4)$. Hence \eqref{eq:geneulRm} becomes
\begin{equation}\label{eq:eulerred}
    |A_1A_2|^2+|A_2A_3|^2+|A_3A_4|^2+|A_4A_1|^2
=
|A_1A_3|^2+|A_2A_4|^2
+
4|M_{13}M_{24}|^2,
\end{equation}
which is precisely Euler's quadrilateral theorem.
\end{proof}

\subsection{Example and Comparison with Earlier Generalizations}

In this subsection we illustrate Theorem \ref{thm:geneul} for $n=6$ vectors in $\mathbb{R}^m$ and compare it with previously known extensions of Euler's theorem.




\begin{example} For $n=6$, Theorem \ref{thm:geneul} yields
\begin{align*}
6\bigl(|A_1A_2|^2+|A_2A_3|^2+\cdots+|A_6A_1|^2\bigr)
&=2(|A_1A_4|^2+|A_2A_5|^2+|A_3A_6|^2) \\
&\quad+4\bigl(|M_{13}M_{24}|^2+|M_{13}M_{25}|^2+|M_{13}M_{26}|^2\\
&\quad+|M_{14}M_{24}|^2+|M_{14}M_{25}|^2+|M_{14}M_{26}|^2+|M_{14}M_{35}|^2\\
&\quad+|M_{14}M_{36}|^2+|M_{15}M_{26}|^2+|M_{15}M_{36}|^2+|M_{15}M_{46}|^2 \\
&\quad+|M_{24}M_{35}|^2+|M_{24}M_{36}|^2+|M_{25}M_{36}|^2 \\
&\quad+|M_{25}M_{46}|^2+|M_{35}M_{46}|^2\bigr).
\end{align*}

For comparison, Dlab and Williams \cite{Dlab-Williams2019} obtained
\begin{align*}
3\bigl(|A_1A_2|^2+|A_2A_3|^2+\cdots+|A_6A_1|^2\bigr)
&=2\bigl(|A_1A_3|^2+|A_2A_4|^2+|A_3A_5|^2+|A_4A_6|^2
+|A_5A_1|^2\\
&\quad+|A_6A_2|^2\bigr)-\bigl(|A_1A_4|^2+|A_2A_5|^2+|A_3A_6|^2
+|A_4A_1|^2\\
&\quad+|A_5A_2|^2+|A_6A_3|^2\bigr)+4\bigl(|B_1B_2|^2+|B_2B_3|^2\\
&\quad+|B_3B_4|^2+|B_4B_5|^2+|B_5B_6|^2+|B_6B_1|^2\bigr),
\end{align*}
where $B_1$ is the midpoint of $A_6A_2$ and, for $i=2,\ldots,6$, $B_i$ is the midpoint of $A_{i-1}A_{i+1}$.

\medskip

The identity of Dlab and Williams and the specialization of Theorem \ref{thm:geneul} for $n=6$ are both Euler-type extensions of Euler's quadrilateral theorem, but they differ substantially in form and scope. The Dlab--Williams identity expresses the sum of the squared side lengths in terms of the lengths of the short and long diagonals together with distances between consecutive midpoint vertices $B_1,\ldots,B_6$. In contrast, Theorem \ref{thm:geneul} involves only the long diagonals $A_1A_4$, $A_2A_5$, and $A_3A_6$, while the correction term is given by a sum of squared distances between pairs of midpoints $M_{ij}$. Moreover, Theorem \ref{thm:geneul} is an algebraic identity valid in every real or complex inner product space and holds uniformly for all $n\ge4$. By comparison, the theorem of Dlab and Williams is formulated for $n$-gons with $n\ge6$ and is derived recursively from their result for $n=5$, which itself is obtained through repeated applications of Euler's quadrilateral theorem. Thus, while both results generalize Euler's theorem, Theorem \ref{thm:geneul} provides a single algebraic framework encompassing all cases $n\ge4$.
\end{example}

\section{Generalized Parallelogram and Apollonius' Identity}\label{sec:cengenapoll}
\noindent
The generalized Euler identity established in the previous section naturally leads to related \,extensions of two classical metric identities: the parallelogram identity and Apollonius' theorem. In this \,section, we derive a generalized parallelogram identity for an even number of vectors and show its geometric interpretation in terms of centroids of alternating point configurations. We then \,establish a generalized Apollonius identity, obtained through a centroid version of the Huygens--Steiner \,formula, and discuss several geometric consequences for finite point sets in Euclidean space.

\subsection{A Generalized Parallelogram Identity}
\label{sec:geneulcentroid}

\begin{theorem}\label{thm:geneulG}
Let $n\geq 4$ be an even integer, and let
$x_1,x_2,\ldots,x_n$ be vectors in an inner product space $V$. Then
\begin{equation}\label{eq:geneulx}
\sum_{i=1}^{n}\|x_i-x_{i+1}\|^2
=
\sum_{\substack{1\le i<j\le n\\ 1<j-i<n-1}}
(-1)^{j-i}\|x_i-x_j\|^2
+
\left\|
\sum_{i=1}^{n}(-1)^{i-1}x_i
\right\|^2,
\end{equation}
where $x_{n+1}=x_1$.
\end{theorem}
\begin{proof}
A direct calculation shows the following.

\begin{itemize}
\item[(a)] The coefficient of $\|x_k\|^2$, for
$k\in\{1,\dots,n\}$, in the expansion of
\[
\sum_{\substack{1\le i<j\le n\\1<j-i<n-1}}
(-1)^{j-i}\|x_i-x_j\|^2
\]
is equal to $1$. Indeed,
\[
\sum_{j=k+2}^{k+n-2}(-1)^{j-k}
=
1,
\]
since $n$ is even and the sum contains $n-3$ terms.

\item[(b)] For $1<j-i<n-1$, the coefficients of
$\operatorname{Re}(x_i,x_j)$ in the expansions of
\[
\sum_{\substack{1\le i<j\le n\\1<j-i<n-1}}
(-1)^{j-i}\|x_i-x_j\|^2
\]
and
\[
\left\|
\sum_{i=1}^{n}(-1)^{i-1}x_i
\right\|^2
\]
are $-2(-1)^{j-i}$ and $2(-1)^{j-i}$, respectively.
\end{itemize}

Therefore, all mixed terms $\operatorname{Re}(x_i,x_j)$ with
$1<j-i<n-1$ cancel. Hence
\begin{align*}
&\sum_{\substack{1\le i<j\le n\\1<j-i<n-1}}
(-1)^{j-i}\|x_i-x_j\|^2
+
\left\|
\sum_{i=1}^{n}(-1)^{i-1}x_i
\right\|^2 \\
&\qquad =
2\sum_{i=1}^{n}\|x_i\|^2
-
2\sum_{i=1}^{n}\operatorname{Re}(x_i,x_{i+1}) \\
&\qquad =
\sum_{i=1}^{n}\|x_i-x_{i+1}\|^2,
\end{align*}
where $x_{n+1}=x_1$.

This completes the proof.
\end{proof}

\begin{corollary}[Douglas \cite{Douglas1981}]\label{cor:geneulG}
Let $n\ge 2$ be an integer, and let
$A_1,A_2,\ldots,A_{2n}$ be (not necessarily distinct) points in
$\mathbb{R}^m$. Let $G_1$ and $G_2$ denote the centroids of the point sets
\[
\{A_1,A_3,\ldots,A_{2n-1}\}
\quad\text{and}\quad
\{A_2,A_4,\ldots,A_{2n}\},
\]
respectively. Then
\begin{equation}\label{eq:geneulG}
\sum_{i=1}^{2n}|A_iA_{i+1}|^2
=
\sum_{\substack{1\le i<j\le 2n\\1<j-i<2n-1}}
(-1)^{j-i}|A_iA_j|^2
+
n^2|G_1G_2|^2,
\end{equation}
where $A_{2n+1}=A_1$.

In particular, for $n=2$, \eqref{eq:geneulG} reduces to Euler's
quadrilateral theorem.
\end{corollary}

\begin{proof}
Let $x_i$ be the position vector of $A_i$ for
$i=1,\ldots,2n$, and let $x_{G_1}$ and $x_{G_2}$ denote the position
vectors of $G_1$ and $G_2$, respectively. Then
\[
\sum_{i=1}^{2n}(-1)^{i-1}x_i
=
n(x_{G_1}-x_{G_2}),
\]
and hence
\[
\left\|
\sum_{i=1}^{2n}(-1)^{i-1}x_i
\right\|^2
=
n^2\|x_{G_1}-x_{G_2}\|^2.
\]
Therefore, \eqref{eq:geneulG} follows immediately from
Theorem \ref{thm:geneulG}.

When $n=2$, we obtain Euler's quadrilateral theorem \eqref{eq:eulerred}.
\end{proof}

\begin{remark}
Corollary \ref{cor:geneulG} may be viewed as a particular extension of Euler's quadrilateral theorem to configurations consisting of an even number of points. Indeed, when $n=2$, the centroids $G_1$ and $G_2$ reduce to the midpoints of the diagonals of a quadrilateral, and \eqref{eq:geneulG} becomes precisely Euler's quadrilateral identity. Thus, \eqref{eq:geneulG} extends Euler's theorem from four points to an arbitrary even number of points by replacing the midpoint term with the squared distance between the centroids of the odd- and even-indexed point sets.
\end{remark}

\begin{corollary}[Generalized Parallelogram Identity]\label{cor:genparid}
Let $n\geq 2$ be an integer, and let
$A_1,A_2,\ldots,A_{2n}$ be (not necessarily distinct) points in
$\mathbb{R}^m$. Let $G_1$ and $G_2$ denote the centroids of the point sets
\[
\{A_1,A_3,\ldots,A_{2n-1}\}
\quad\text{and}\quad
\{A_2,A_4,\ldots,A_{2n}\},
\]
respectively. Then
\begin{equation}\label{eq:genparid}
\sum_{i=1}^{2n}|A_iA_{i+1}|^2
=
\sum_{\substack{1\le i<j\le 2n\\1<j-i<2n-1}}
(-1)^{j-i}|A_iA_j|^2
\end{equation}
if and only if $G_1=G_2$ (see Figure \ref{fig:genparid}).
\end{corollary}
\begin{figure}
    \centering
\begin{tikzpicture}
\clip(-5,-5) rectangle (10,10);
\fill[line width=1.pt,color=qqzzqq,fill=qqzzqq,fill opacity=0.15] (3.708033991789106,4.912717585742084) -- (4.,0.) -- (8.912717585742083,0.29196600821089325) -- (8.620751577531191,5.204683593952977) -- cycle;
\fill[line width=1.pt,color=qqzzqq,fill=qqzzqq,fill opacity=0.15] (4.,0.) -- (0.,3.) -- (-3.,-1.) -- (1.,-4.) -- cycle;
\fill[line width=1.pt,color=qqzzqq,fill=qqzzqq,fill opacity=0.15] (0.,3.) -- (3.708033991789106,4.912717585742084) -- (1.7953164060470224,8.62075157753119) -- (-1.9127175857420835,6.708033991789106) -- cycle;
\fill[line width=1.pt,color=qqzzqq,fill=qqzzqq,fill opacity=0.15] (1.,1.) -- (1.,5.) -- (-3.,5.) -- (-3.,1.) -- cycle;
\fill[line width=1.pt,color=qqzzqq,fill=qqzzqq,fill opacity=0.15] (5.725767782492991,1.910924856620419) -- (1.,1.) -- (1.9109248566204182,-3.72576778249299) -- (6.636692639113409,-2.814842925872572) -- cycle;
\fill[line width=1.pt,color=qqzzqq,fill=qqzzqq,fill opacity=0.15] (1.,5.) -- (5.725767782492991,1.910924856620419) -- (8.814842925872572,6.636692639113409) -- (4.089075143379582,9.72576778249299) -- cycle;
\fill[line width=1.pt,color=zzttqq,fill=zzttqq,fill opacity=0.15] (0.,3.) -- (1.,5.) -- (-1.,6.) -- (-2.,4.) -- cycle;
\fill[line width=1.pt,color=zzttqq,fill=zzttqq,fill opacity=0.15] (1.,1.) -- (0.,3.) -- (-2.,2.) -- (-1.,0.) -- cycle;
\fill[line width=1.pt,color=zzttqq,fill=zzttqq,fill opacity=0.15] (4.,0.) -- (1.,1.) -- (0.,-2.) -- (3.,-3.) -- cycle;
\fill[line width=1.pt,color=zzttqq,fill=zzttqq,fill opacity=0.15] (5.725767782492991,1.910924856620419) -- (4.,0.) -- (5.910924856620419,-1.7257677824929902) -- (7.636692639113409,0.185157074127428) -- cycle;
\fill[line width=1.pt,color=zzttqq,fill=zzttqq,fill opacity=0.15] (3.708033991789106,4.912717585742084) -- (5.725767782492991,1.910924856620419) -- (8.727560511614655,3.9286586473243035) -- (6.709826720910771,6.930451376445969) -- cycle;
\fill[line width=1.pt,color=zzttqq,fill=zzttqq,fill opacity=0.15] (1.,5.) -- (3.708033991789106,4.912717585742084) -- (3.7953164060470224,7.620751577531188) -- (1.0872824142579176,7.708033991789105) -- cycle;
\fill[line width=1.pt,color=blue,fill=blue,fill opacity=0.15] (4.,0.) -- (1.,5.) -- (-4.,2.) -- (-1.,-3.) -- cycle;
\fill[line width=1.pt,color=blue,fill=blue,fill opacity=0.15] (5.725767782492991,1.910924856620419) -- (0.,3.) -- (-1.089075143379581,-2.72576778249299) -- (4.636692639113409,-3.8148429258725716) -- cycle;
\fill[line width=1.pt,color=blue,fill=blue,fill opacity=0.15] (3.708033991789106,4.912717585742084) -- (1.,1.) -- (4.912717585742083,-1.7080339917891054) -- (7.620751577531189,2.2046835939529776) -- cycle;
\draw [line width=1.pt] (1.,5.)-- (0.,3.);
\draw [line width=1.pt] (0.,3.)-- (1.,1.);
\draw [line width=1.pt] (1.,1.)-- (4.,0.);
\draw [line width=1.pt] (4.,0.)-- (5.725767782492991,1.910924856620419);
\draw [line width=1.pt] (5.725767782492991,1.910924856620419)-- (3.708033991789106,4.912717585742084);
\draw [line width=1.pt] (3.708033991789106,4.912717585742084)-- (1.,5.);
\draw [line width=2.4pt,color=blue] (1.,5.)-- (1.,1.);
\draw [line width=1.pt] (1.,5.)-- (4.,0.);
\draw [line width=2.4pt,color=blue] (1.,5.)-- (5.725767782492991,1.910924856620419);
\draw [line width=2.4pt,color=qqzzqq] (0.,3.)-- (4.,0.);
\draw [line width=1.pt] (0.,3.)-- (5.725767782492991,1.910924856620419);
\draw [line width=2.4pt,color=qqzzqq] (0.,3.)-- (3.708033991789106,4.912717585742084);
\draw [line width=2.4pt,color=blue] (1.,1.)-- (5.725767782492991,1.910924856620419);
\draw [line width=1.pt] (1.,1.)-- (3.708033991789106,4.912717585742084);
\draw [line width=2.4pt,color=qqzzqq] (4.,0.)-- (3.708033991789106,4.912717585742084);
\draw [line width=1.pt,color=qqzzqq] (4.,0.)-- (8.912717585742083,0.29196600821089325);
\draw [line width=1.pt,color=qqzzqq] (8.912717585742083,0.29196600821089325)-- (8.620751577531191,5.204683593952977);
\draw [line width=1.pt,color=qqzzqq] (8.620751577531191,5.204683593952977)-- (3.708033991789106,4.912717585742084);
\draw [line width=1.pt,color=qqzzqq] (4.,0.)-- (0.,3.);
\draw [line width=1.pt,color=qqzzqq] (0.,3.)-- (-3.,-1.);
\draw [line width=1.pt,color=qqzzqq] (-3.,-1.)-- (1.,-4.);
\draw [line width=1.pt,color=qqzzqq] (1.,-4.)-- (4.,0.);
\draw [line width=1.pt,color=qqzzqq] (0.,3.)-- (3.708033991789106,4.912717585742084);
\draw [line width=1.pt,color=qqzzqq] (3.708033991789106,4.912717585742084)-- (1.7953164060470224,8.62075157753119);
\draw [line width=1.pt,color=qqzzqq] (1.7953164060470224,8.62075157753119)-- (-1.9127175857420835,6.708033991789106);
\draw [line width=1.pt,color=qqzzqq] (-1.9127175857420835,6.708033991789106)-- (0.,3.);
\draw [line width=1.pt,color=qqzzqq] (1.,1.)-- (1.,5.);
\draw [line width=1.pt,color=qqzzqq] (1.,5.)-- (-3.,5.);
\draw [line width=1.pt,color=qqzzqq] (-3.,5.)-- (-3.,1.);
\draw [line width=1.pt,color=qqzzqq] (-3.,1.)-- (1.,1.);
\draw [line width=1.pt,color=qqzzqq] (5.725767782492991,1.910924856620419)-- (1.,1.);
\draw [line width=1.pt,color=qqzzqq] (1.,1.)-- (1.9109248566204182,-3.72576778249299);
\draw [line width=1.pt,color=qqzzqq] (1.9109248566204182,-3.72576778249299)-- (6.636692639113409,-2.814842925872572);
\draw [line width=1.pt,color=qqzzqq] (6.636692639113409,-2.814842925872572)-- (5.725767782492991,1.910924856620419);
\draw [line width=1.pt,color=qqzzqq] (1.,5.)-- (5.725767782492991,1.910924856620419);
\draw [line width=1.pt,color=qqzzqq] (5.725767782492991,1.910924856620419)-- (8.814842925872572,6.636692639113409);
\draw [line width=1.pt,color=qqzzqq] (8.814842925872572,6.636692639113409)-- (4.089075143379582,9.72576778249299);
\draw [line width=1.pt,color=qqzzqq] (4.089075143379582,9.72576778249299)-- (1.,5.);
\draw [line width=1.pt,color=zzttqq] (0.,3.)-- (1.,5.);
\draw [line width=1.pt,color=zzttqq] (1.,5.)-- (-1.,6.);
\draw [line width=1.pt,color=zzttqq] (-1.,6.)-- (-2.,4.);
\draw [line width=1.pt,color=zzttqq] (-2.,4.)-- (0.,3.);
\draw [line width=1.pt,color=zzttqq] (1.,1.)-- (0.,3.);
\draw [line width=1.pt,color=zzttqq] (0.,3.)-- (-2.,2.);
\draw [line width=1.pt,color=zzttqq] (-2.,2.)-- (-1.,0.);
\draw [line width=1.pt,color=zzttqq] (-1.,0.)-- (1.,1.);
\draw [line width=1.pt,color=zzttqq] (4.,0.)-- (1.,1.);
\draw [line width=1.pt,color=zzttqq] (1.,1.)-- (0.,-2.);
\draw [line width=1.pt,color=zzttqq] (0.,-2.)-- (3.,-3.);
\draw [line width=1.pt,color=zzttqq] (3.,-3.)-- (4.,0.);
\draw [line width=1.pt,color=zzttqq] (5.725767782492991,1.910924856620419)-- (4.,0.);
\draw [line width=1.pt,color=zzttqq] (4.,0.)-- (5.910924856620419,-1.7257677824929902);
\draw [line width=1.pt,color=zzttqq] (5.910924856620419,-1.7257677824929902)-- (7.636692639113409,0.185157074127428);
\draw [line width=1.pt,color=zzttqq] (7.636692639113409,0.185157074127428)-- (5.725767782492991,1.910924856620419);
\draw [line width=1.pt,color=zzttqq] (3.708033991789106,4.912717585742084)-- (5.725767782492991,1.910924856620419);
\draw [line width=1.pt,color=zzttqq] (5.725767782492991,1.910924856620419)-- (8.727560511614655,3.9286586473243035);
\draw [line width=1.pt,color=zzttqq] (8.727560511614655,3.9286586473243035)-- (6.709826720910771,6.930451376445969);
\draw [line width=1.pt,color=zzttqq] (6.709826720910771,6.930451376445969)-- (3.708033991789106,4.912717585742084);
\draw [line width=1.pt,color=zzttqq] (1.,5.)-- (3.708033991789106,4.912717585742084);
\draw [line width=1.pt,color=zzttqq] (3.708033991789106,4.912717585742084)-- (3.7953164060470224,7.620751577531188);
\draw [line width=1.pt,color=zzttqq] (3.7953164060470224,7.620751577531188)-- (1.0872824142579176,7.708033991789105);
\draw [line width=1.pt,color=zzttqq] (1.0872824142579176,7.708033991789105)-- (1.,5.);
\draw [line width=1.pt,color=blue] (4.,0.)-- (1.,5.);
\draw [line width=1.pt,color=blue] (1.,5.)-- (-4.,2.);
\draw [line width=1.pt,color=blue] (-4.,2.)-- (-1.,-3.);
\draw [line width=1.pt,color=blue] (-1.,-3.)-- (4.,0.);
\draw [line width=1.pt,color=blue] (5.725767782492991,1.910924856620419)-- (0.,3.);
\draw [line width=1.pt,color=blue] (0.,3.)-- (-1.089075143379581,-2.72576778249299);
\draw [line width=1.pt,color=blue] (-1.089075143379581,-2.72576778249299)-- (4.636692639113409,-3.8148429258725716);
\draw [line width=1.pt,color=blue] (4.636692639113409,-3.8148429258725716)-- (5.725767782492991,1.910924856620419);
\draw [line width=1.pt,color=blue] (3.708033991789106,4.912717585742084)-- (1.,1.);
\draw [line width=1.pt,color=blue] (1.,1.)-- (4.912717585742083,-1.7080339917891054);
\draw [line width=1.pt,color=blue] (4.912717585742083,-1.7080339917891054)-- (7.620751577531189,2.2046835939529776);
\draw [line width=1.pt,color=blue] (7.620751577531189,2.2046835939529776)-- (3.708033991789106,4.912717585742084);
\begin{scriptsize}
\draw [fill=ududff] (1.,5.) circle (2.0pt);
\draw[color=black] (0.75,5.5) node {$A_1$};
\draw [fill=ududff] (0.,3.) circle (2.0pt);
\draw[color=black] (-0.5,3.05) node {$A_2$};
\draw [fill=blue] (1.,1.) circle (2.0pt);
\draw[color=black] (0.6,0.55) node {$A_3$};
\draw [fill=ududff] (4.,0.) circle (2.0pt);
\draw[color=black] (4.1,-0.3) node {$A_4$};
\draw [fill=ududff] (5.725767782492991,1.910924856620419) circle (2.0pt);
\draw[color=black] (6.1,1.9) node {$A_5$};
\draw [fill=ududff] (3.708033991789106,4.912717585742084) circle (2.0pt);
\draw[color=black] (3.9,5.35) node {$A_6$};
\draw [fill=wewdxt] (2.58,2.64) circle (2.0pt);
\draw[color=black] (2.8,2.9) node {$G_1$};
\draw[color=black] (2.4,2.2) node {$G_2$};
\draw [fill=wewdxt] (8.912717585742083,0.29196600821089325) circle (2.0pt);
\draw [fill=wewdxt] (8.620751577531191,5.204683593952977) circle (2.0pt);
\draw [fill=wewdxt] (-3.,-1.) circle (2.0pt);
\draw [fill=wewdxt] (1.,-4.) circle (2.0pt);
\draw [fill=wewdxt] (1.7953164060470224,8.62075157753119) circle (2.0pt);
\draw [fill=wewdxt] (-1.9127175857420835,6.708033991789106) circle (2.0pt);
\draw [fill=wewdxt] (-3.,5.) circle (2.0pt);
\draw [fill=wewdxt] (-3.,1.) circle (2.0pt);
\draw [fill=wewdxt] (1.9109248566204182,-3.72576778249299) circle (2.0pt);
\draw [fill=wewdxt] (6.636692639113409,-2.814842925872572) circle (2.0pt);
\draw [fill=wewdxt] (8.814842925872572,6.636692639113409) circle (2.0pt);
\draw [fill=wewdxt] (4.089075143379582,9.72576778249299) circle (2.0pt);
\draw [fill=wewdxt] (-1.,6.) circle (2.0pt);
\draw [fill=wewdxt] (-2.,4.) circle (2.0pt);
\draw [fill=wewdxt] (-2.,2.) circle (2.0pt);
\draw [fill=wewdxt] (-1.,0.) circle (2.0pt);
\draw [fill=wewdxt] (0.,-2.) circle (2.0pt);
\draw [fill=wewdxt] (3.,-3.) circle (2.0pt);
\draw [fill=wewdxt] (5.910924856620419,-1.7257677824929902) circle (2.0pt);
\draw [fill=wewdxt] (7.636692639113409,0.185157074127428) circle (2.0pt);
\draw [fill=wewdxt] (8.727560511614655,3.9286586473243035) circle (2.0pt);
\draw [fill=wewdxt] (6.709826720910771,6.930451376445969) circle (2.0pt);
\draw [fill=wewdxt] (3.7953164060470224,7.620751577531188) circle (2.0pt);
\draw [fill=wewdxt] (1.0872824142579176,7.708033991789105) circle (2.0pt);
\draw [fill=wewdxt] (-4.,2.) circle (2.0pt);
\draw [fill=wewdxt] (-1.,-3.) circle (2.0pt);
\draw [fill=wewdxt] (-1.089075143379581,-2.72576778249299) circle (2.0pt);
\draw [fill=wewdxt] (4.636692639113409,-3.8148429258725716) circle (2.0pt);
\draw [fill=wewdxt] (4.912717585742083,-1.7080339917891054) circle (2.0pt);
\draw [fill=wewdxt] (7.620751577531189,2.2046835939529776) circle (2.0pt);
\end{scriptsize}
\end{tikzpicture}
    \caption{Illustration of the generalized parallelogram identity \eqref{eq:genparid} for $n=6$ in $\mathbb{R}^m$. When the centroids $G_1$ and $G_2$ of triangles $\triangle A_1A_3A_5$ and $\triangle A_2A_4A_6$, \,respectively, are equal, the total area of the brown squares on the six sides $\overline{A_1A_2}$, $\overline{A_2A_3}$, \dots, $\overline{A_5A_6}$, and $\overline{A_6A_1}$ equal the combined areas of the six green squares on the short diagonals $\overline{A_1A_3}$, $\overline{A_2A_4}$, \dots, $\overline{A_5A_1}$ and $\overline{A_6A_2}$ minus the combined \,areas of the three blue squares on the long diagonals $\overline{A_1A_4}$, $\overline{A_2A_5}$ and $\overline{A_3A_6}$. The squares are drawn schematically; they need not lie in the same plane, since the points $A_1,A_2, \dots, A_6$ are not necessarily coplanar.}
    \label{fig:genparid}
\end{figure}

\begin{remark}
Corollary \ref{cor:genparid} recovers, in geometric form, a result of Douglas \cite{Douglas1981} for an even number of vectors in $\mathbb R^m$. 
\end{remark}

The classical Apollonius theorem expresses the sum of the squared distances from a point to the endpoints of a segment in terms of the midpoint of the segment. For a finite number of vectors, the role of the midpoint is naturally replaced by the centroid. In this section we derive a centroid identity in an arbitrary inner product space and obtain several geometric consequences for point configurations in $\mathbb{R}^m$.
\begin{theorem}[Huygens–Steiner Identity]\label{thm:huy-stei}
    Let $x_1,x_2,\ldots,x_n$ be vectors in an inner product space $V$, and
\begin{equation*}
        x_{G}:=\frac{1}{n}\sum_{k=1}^n x_k,
\end{equation*}
be their centroid. If $x\in V$ be an arbitrary vector, then
\begin{equation}\label{eq:huy-stei}
    \sum_{k=1}^{n}\Vert x_k-x\Vert^2=\sum_{k=1}^{n}\Vert x_k-x_G\Vert^2+n \Vert x-x_G\Vert^2.
\end{equation}
\end{theorem}
\begin{proof} We take $X=x_k-x_G$ and $Y=x-x_G$ in the following identity:
\begin{equation*}
    \Vert X-Y\Vert^2= \Vert X \Vert^2+\Vert Y \Vert^2-2\mathrm{Re} (X, Y)
\end{equation*}
and summing over $k$, to obtain
    \begin{equation*}
        \sum_{k=1}^{n}\Vert x_k-x\Vert^2=\sum_{k=1}^{n}\Vert x_k-x_G\Vert^2+n \Vert x-x_G\Vert^2-2\mathrm{Re}\left(\sum_{k=1}^{n}x_k-nx_G,x-x_G \right),
    \end{equation*}
which  gives \eqref{eq:huy-stei}, since the last term vanishes as $\sum_{k=1}^{n}x_k=nx_G$.
\end{proof}

\begin{corollary}\label{cor:xkxgn}
    Let $x_1,x_2,\ldots,x_n$ be vectors in an inner product space $V$, and $x_{G}$ be their centroid. Then
\begin{equation}\label{eq:xkxgn}
    \sum_{1\leq k<l\leq n}\Vert x_k-x_l\Vert^2=n\sum_{k=1}^{n}\Vert x_k-x_G\Vert^2.
\end{equation}
\end{corollary}
\begin{proof}
    Applying \eqref{eq:huy-stei} for $x=x_l$ and summing over $l=1,...,n$, we obtain
\begin{equation}\label{eq:xkxgnpr}
    \sum_{l=1}^{n}\sum_{k=1}^{n}\Vert x_k-x_l\Vert^2=n\sum_{k=1}^{n}\Vert x_k-x_G\Vert^2+n \sum_{l=1}^{n}\Vert x_l-x_G\Vert^2
\end{equation}
Since
\begin{equation*}
    \sum_{l=1}^{n}\sum_{k=1}^{n}\Vert x_k-x_l\Vert^2=2\sum_{1\leq k<l\leq n}\Vert x_k-x_l\Vert^2.
\end{equation*}
equation \eqref{eq:xkxgnpr} reduces to \eqref{eq:xkxgn}, completing the proof. 
\end{proof}

\begin{remark}
When $V=\mathbb{R}^m$ and the vectors $x_1,\ldots,x_n$ are interpreted as the position vectors of points
$A_1,\ldots,A_n$ with centroid $G$, identity \eqref{eq:xkxgn} becomes
\[
\sum_{1\le k<l\le n}|A_kA_l|^2
=
n\sum_{k=1}^n |A_kG|^2.
\]
This classical relation is often attributed to Leibniz and may be viewed as a centroid version of Apollonius' theorem. It expresses the total pairwise squared distance of a finite point configuration in terms of the squared distances from its centroid.
\end{remark}

\begin{corollary}[Generalized Apollonius]\label{cor:genapoll}
Let $x_1,x_2,\ldots,x_n$ be vectors in an inner product space $V$, and let $x_G$ denote their centroid. Then, for every $x\in V$,
\begin{equation*}
    \sum_{k=1}^{n}\Vert x_k-x\Vert^2
    =
    \frac{1}{n}\sum_{1\leq k< l\leq n}\Vert x_k-x_l\Vert^2
    +
    n\Vert x-x_G\Vert^2.
\end{equation*}
\end{corollary}

\begin{proof}
Substituting \eqref{eq:xkxgn} into \eqref{eq:huy-stei} yields the result.
\end{proof}

\begin{remark}\label{rem:genapollrm}
Let $A_1,A_2,\ldots,A_n$ be points in $\mathbb{R}^m$ with centroid $G$. Then, for every point $P\in\mathbb{R}^m$,
\begin{equation}\label{eq:genapollo}
    \sum_{k=1}^{n}|A_kP|^2
    =
    \frac{1}{n}\sum_{1\le k<l\le n}|A_kA_l|^2
    +
    n|PG|^2.
\end{equation}
Indeed, if $x_k$ and $x$ denote the position vectors of $A_k$ and $P$, respectively, then equation \eqref{eq:genapollo} follows immediately from Corollary \ref{cor:genapoll}.
\end{remark}

\begin{corollary}\label{cor:sqsumsOG}
    Let $A_1,A_2,\ldots,A_n$ be $n$ points in $\mathbb{R}^m$ and $G$ be their centroid. If $O\in \mathbb{R}^m$ such that $|OA_k|=R$ for $k=1,\ldots n$, then
                \begin{equation*}
                    \sum_{1\leq k<l\leq n}|A_kA_l|^2=n^2(R^2-|OG|^2).
                \end{equation*}
\end{corollary}
\begin{proof} Take $P=O$ in equation \eqref{eq:genapollo} and use the assumption $|OA_k|=R$ for $k=1,\ldots n$.
\end{proof}

\begin{proposition}\label{prop:disverremcent}
    Let $x_1,x_2,\ldots,x_n$ be vectors in an inner product space $V$, and let $x_G$ denote their centroid. For each $k$, let $x_{G_k}$ denote the centroid of the vectors other than $x_k$. Then
\begin{equation*}
    \sum_{k=1}^{n}\Vert x_k-x_{G_k}\Vert^2
    =
    \frac{n^2}{(n-1)^2}\left(\sum_{k=1}^{n}\Vert x_k\Vert^2-
    n\Vert x_G\Vert^2\right).
\end{equation*}
\end{proposition}

\begin{proof}
    \begin{align*}
        \Vert x_k-x_{G_k}\Vert^2&=\left\Vert x_k-\frac{nx_{G}-x_k}{n-1}\right\Vert^2\\
        &=\frac{n^2}{(n-1)^2} \Vert x_k-x_G\Vert^2
    \end{align*}
Summing over $k$, we get 
    \begin{align*}
        \sum_{k=1}^{n}\Vert x_k-x_{G_k}\Vert^2=\frac{n^2}{(n-1)^2} \sum_{k=1}^{n}\Vert x_k-x_G\Vert^2.
    \end{align*}
Now using the Huygens--Steiner identity \eqref{eq:huy-stei} with $x=0$, we obtain the result.
\end{proof}

\begin{corollary}
Let $A_1,A_2,\ldots,A_n$ be points in the Euclidean plane, and let $G$ denote their centroid. For each $k$, let $G_{A_k}$ denote the centroid of the points other than $A_k$. Then
If $A_1,A_2,\ldots,A_n$ lie on a circle of radius $R$, then
\[
\sum_{k=1}^{n}\vert A_kG_{A_k}\vert^2
=
\frac{n^3}{(n-1)^2}
\left(
R^2-|OG|^2
\right).
\]
\end{corollary}

\begin{proof}
Since $\sum_{k=1}^{n}|OA_k|^2=nR^2$, the result follows immediately from Proposition \ref{prop:disverremcent}.
\end{proof}

\section{Concluding Remarks}\label{sec:conrem}

Euler's quadrilateral theorem is naturally viewed as an identity in an inner product space rather than merely a metric property of planar quadrilaterals. In this paper, we established a generalized Euler identity valid for every finite collection of $n\geq4$ vectors in a real or complex inner product space. The classical theorem of Euler is recovered when $n=4$, while Douglas's theorem for an even number of vectors appears as a natural centroid formulation within the same framework.

Theorem \ref{thm:geneul} provides a unified extension of Euler's theorem for all $n\geq4$, thereby placing \,several previously known generalizations into a single algebraic setting. The corresponding \,geometric \,formulation, given in Corollary \ref{cor:geneulRm}, extends the midpoint structure underlying Euler's original \,theorem and expresses the identity in terms of distances between midpoint configurations.

In addition, we derived generalized versions of the parallelogram and Apollonius identities and interpreted them through centroid relations for finite point sets. Together, these results illustrate how classical metric identities can be unified and extended by elementary methods from the theory of inner product spaces, revealing further connections between geometric configurations and their underlying algebraic structure.

\medskip
\noindent \textbf{DISCLOSURE STATEMENT}. The author reports there is no conflict of interest.

\vspace{0.5cm}
\noindent
\thanks{\textbf{Acknowledgments}. The author acknowledges the use of computer algebra systems \textit{Mathematica} and \textit{GeoGebra} for various symbolic computations.}

\bibliographystyle{amsplain}  
\bibliography{references}

\end{document}